\documentclass[11pt]{amsart}
\usepackage{graphicx}
\usepackage{amsmath}
\usepackage{amsthm}
\usepackage{listings}
\usepackage{hyperref}
\usepackage{tikz}
\usepackage{float}
\usepackage{amssymb}
\usepackage{color}                    

\makeatletter
\newtheorem*{rep@theorem}{\rep@title}
\newcommand{\newreptheorem}[2]{%
\newenvironment{rep#1}[1]{%
 \def\rep@title{#2 \ref{##1}}%
 \begin{rep@theorem}}%
 {\end{rep@theorem}}}
\makeatother

\newtheorem{theorem}{Theorem}[section]
\theoremstyle{definition}
\newreptheorem{theorem}{Theorem}

\newtheorem{corollary}{Corollary}[theorem]
\newtheorem{lemma}[theorem]{Lemma}
\newtheorem*{remark}{Remark}

\def\sm{{\mathcal M}}
\def\N{{\mathbb N}}
\def\Q{{\mathbb Q}}
\def\R{{\mathbb R}}
\def\mult{\operatorname{mult}\,}

\title[Multiplicities and Quadratic Representations]{Multiplicities of eigenvalues and quadratic representations of integers}
\author{Siqi Fu and Andrew Pendleton}

\address{Department of Mathematical Sciences,
Rutgers University-Camden, Camden, NJ 08102}
\email{sfu@camden.rutgers.edu}
\email{ap3179@scarletmail.rutgers.edu, ajpendleton.math@gmail.com} 

\thanks{This project was partially supported by NSF grant DMS-2055538.}

\begin{document}

\maketitle
\begin{abstract}
    We study the set $\sm$ of all multiplicities of non-zero eigenvalues for the Laplace operator on a two-dimensional rectangle or torus. We show that for a rectangle with the side length ratio $r$, $\sm=\N$, the set of all positive integers, if and only if $r^2$ is rational. 
    For a torus whose generating vectors have a length ratio $r$ and the angle between them $\theta$, we show that $\sm$ is an infinite set if and only if both $r\cos\theta$ and $r^2$ are rational. In this case, $\sm=2\mathbb{N}, 4\mathbb{N}$, or $6\mathbb{N}$, and we obtain a characterization for each of these cases in terms of $r\cos\theta$ and $r^2$. In the case when at least one of $r\cos\theta$ or $r^2$ is irrational, we show that $\sm=\{2\}$ or $\{2, 4\}$, and obtain a characterization for these cases. We prove these results by studying the number of integral lattice points on dilated ellipses.


\smallskip

\end{abstract}

\section{Introduction} \label{intro}

In this paper, we study two seemingly unrelated subjects: the multiplicities of eigenvalues of the Laplace operator and quadratic representations of integers. For a domain in the plane, the eigenvalues of the Laplace operator subject to the Dirichlet boundary condition represent the frequencies of the vibrations of the drumhead. The multiplicity of an eigenvalue represents the number of modes associated with the corresponding frequency. The interplay between the spectral properties of the Laplace operator and the geometry of the underlying space has been studied extensively since Mark Kac posed his famous problem ``Can One Hear the Shape of a Drum?" in 1966 (\cite{Kac1966}). Here, we are interested in the set $\sm$ of all multiplicities of (non-zero) eigenvalues and how its structure relates to the geometric and algebraic properties of the underlying space. 

The classical Courant's Nodal Domain Theorem states that the number of nodal domains of an eigenfunction associated with the $k^{\rm th}$-eigenvalue  is at most $k$.  As a consequence, the first eigenvalue on a (connected) domain is always simple. For smoothly bounded planar domains, it was shown by S.-Y. Cheng~\cite{Cheng} that the multiplicity of the second eigenvalue is less than or equal to $3$ and by Nadirashvili~\cite{Nad} that the multiplicity of the $k^{\text{th}}$ eigenvalue is less than or equal to $2k-1$ for $k\ge 3$ (see the recent preprint \cite{BerardHelffer} and references therein for an extensive discussion of relevant results). 

In general, it is difficult to explicitly compute eigenvalues or determine their multiplicities. For string vibrations, all eigenvalues are simple. On a circular planar drumhead, the first eigenvalue is simple while all other eigenvalues have multiplicities $2$. For the unit ball in $\R^n$, $n\ge 2$, the set of multiplicities of eigenvalues for the Dirichlet Laplacian is identical to the set of all dimensions of the spaces of spherical harmonics of degree $k$ ($k\ge 0$), which is given by $\sm(B_n)=\{1, n, C(k+n-1, n-1)-C(k+n-3, n-1) \mid k\ge 2\}$, where $C(\cdot, \cdot)$ denotes the combination number. The set of multiplicities for the Laplace operator on the unit sphere in $\R^n$ is the same as that of the unit ball (see \cite[Chapter IV]{SW} for relevant material on spherical harmonics).

For planar domains,  explicit formulas for the eigenvalues are known only for a few cases, such as circles, rectangles, equilateral triangles, hemi-equilateral triangles, and isosceles right triangles (see \cite{McCartin}).  In \cite{Pinsky}, Pinsky obtained a formula for multiplicities of eigenvalues on the equilateral triangles and show that the set of multiplicities of eigenvalues is $\mathbb{N}$. In \cite{Fu2025-ph},  the authors studied the set of multiplicities of eigenvalues for the Laplace operator on rectangles. They observed that on the rectangle $\Omega_{a, b}=[0,a]\times[0,b]$, $\sm=\{1\}$ if $(a/b)^2\notin\mathbb{Q}$, and showed that squares also satisfy property ($M$),  as a simple consequence of Legendre's formula on the number of ways an integer can be represented as a sum of squares of integers. (A domain is said to satisfy property ($M$) if the set of multiplicities of eigenvalues for the Dirichlet Laplacian is $\N$ \cite{Fu2025-ph}.)  In \cite{Heimrath}, Heimrath obtained formulas for multiplicities of eigenvalues on rectangles such that $(a/b)^2\in \{2, 3, 7\}$ through  generalizing Legendre's formula and as a consequence, showed that these rectangles satisfy property ($M$). The problem of whether or not $\Omega_{a, b}$ satisfying property ($M$) is characterized by $(a/b)^2\in\Q$ was left open in \cite{Fu2025-ph} and \cite{Heimrath}. The first result in this paper answers this question affirmatively:

\begin{theorem}\label{M}
    The rectangle $\Omega_{a, b}=[0,a]\times[0,b]$ satisfies property (M) if and only if $(a/b)^2$ is rational.
\end{theorem}

Our proof differs from \cite{Heimrath} in that we do not provide formulas for the number of lattice points on general ellipses. Instead, we prove the existence of some eigenvalue with arbitrary multiplicity, which is a direct consequence of the following theorem:

\begin{theorem}\label{Q}
    For any $m,n\in\mathbb{N}$ coprime such that $mn>3$, there exists some prime $p$ such that \[p^{2k-1}=mx^2+ny^2\] has exactly $k$ distinct solutions $(x,y)\in\mathbb{N}^2$.
\end{theorem} 

We now turn our attention to the multiplicities of eigenvalues for the Laplace operator on 2-dimensional tori. Let $L$ be the lattice in $\R^2$ generated by linearly independent vectors $\vec{a}$ and $\vec{b}$. Let $T=\R^2/L$.  Let $r=|\vec{a}|/|\vec{b}|$ and $\theta\in (0, \ \pi/2]$ be the angle between $\vec{a}$ and $\vec{b}$. Our main result regarding the set $\sm(T)$ of multiplicities of eigenvalues for the Laplace operator on $T$ is:

\begin{theorem}\label{N} The multiplicity set $\mathcal{M}(T)$ of the Laplace operator on torus $T$ is infinite if and only if $r^2\in\mathbb{Q}$ and $r\cos{\theta}\in \mathbb{Q}$. In this case,  
\[
\mathcal{M}(T)=
\begin{cases}
    6\mathbb{N}&\text{if and only if}\; \Delta=-3;\\
    4\mathbb{N}&\text{if and only if}\; \Delta=-4;\\
    2\mathbb{N}&\text{if and only if}\; \Delta\neq-3,-4;\\
\end{cases}
\]
where $\Delta$ is the discriminant defined by
\[
\Delta = -4\left(\dfrac{r\cdot\beta\delta}{\gcd(\beta,\delta)}\right)^2\sin^2\theta
\]
with
\[
\beta:=\min\{n\in\mathbb{N}:2nr\cos\theta\in\mathbb{Z}_{\geq 0}\} \quad\text{and}\quad \delta:=\min\{n\in\mathbb{N}:nr^2\in\mathbb{Z}_{\geq 0}\}.
\]
The multiplicity set $\mathcal{M}(T)$ is finite if and only if  either $r^2\not\in\mathbb{Q}$ or $r\cos\theta\not\in\mathbb{Q}$. In this case,  we have 
$\mathcal{M}(T)=\{2\}$ if and only if one of the followings holds: 
\begin{itemize}
\item[$(1)$] $r\cos\theta\not\in\mathbb{Q}$, $r^2\in\mathbb{Q}$ but $r\not\in\mathbb{Q}$;
\item[$(2)$] $r\cos\theta\not\in\mathbb{Q}$, $r^2\not\in\mathbb{Q}$, and $\{1, r\cos\theta, r^2\}$ is linearly independent over $\mathbb{Q}$;
\item[$(3)$] $r\cos\theta\not\in\mathbb{Q}$, $r^2\not\in\mathbb{Q}$, and $r\cos\theta=\alpha r^2+\beta$ for some $\alpha, \beta\in\mathbb{Q}$ such that $\alpha^2+\beta$ is not a square of a rational number.
\end{itemize}
Moreover, $\mathcal{M}(T)=\{2, 4\}$ if and only if one of the followings holds:
\begin{itemize}
\item[$(1)$] $r\cos\theta\in\mathbb{Q}$ and $r^2\not\in\mathbb{Q}$;
\item[$(2)$] $r\cos\theta\not\in\mathbb{Q}$ and $r\in\mathbb{Q}$;
\item[$(3)$] $r\cos\theta\not\in\mathbb{Q}$, $r^2\not\in\mathbb{Q}$, and $r\cos\theta=\alpha r^2+\beta$ for some $\alpha, \beta\in\mathbb{Q}$ such that $\alpha^2+\beta$ is a square of a rational number.
\end{itemize}
\end{theorem}

As in the proof of Theorem~\ref{M}, the main ingredient in the proof of Theorem~\ref{N} when both $r^2$ and $r\cos\theta$ are rational numbers is  
quadratic representations of integers. (see Section~\ref{s:prelim} below for relevant material.) On that account, our main result is:

\begin{theorem}\label{Q2}
    Let $Q(x, y)=ax^2+bxy+cy^2$ be a primitive positive-definite binary quadratic form. Then the proper representation counting function \[r_Q^+:\mathbb{Z}_{\geq 0}\twoheadrightarrow \mathbb{Z}_{\geq 0}\] is surjective.
\end{theorem}

This paper is organized as follows. Section~\ref{s:prelim} contains preliminaries on eigenvalues for rectangles and tori, as well as rudiments on quadratic representations of integers. We prove Theorem~\ref{Q} and Theorem~\ref{M} in Section~\ref{s:rectangles}, and Theorem~\ref{Q2} and Theorem~\ref{N} in Section~\ref{s:tori}.  We have made an effort to provide a presentation that is accessible to readers with limited backgrounds in differential equations or number theory.

\section{Preliminaries}\label{s:prelim}
\subsection{The Laplace operator on Rectangles and Tori}\label{s:prelim:delta}

Consider an elastic and homogeneous drumhead with density $\rho$ and tension $T$ stretched over a
frame represented as a domain $\Omega\subset\R^2$ in the
$xy$-plane.  Let $u(x, y, t)$ be the vertical displacement of the
membrane from the equilibrium position, and assume that the
horizontal displacement is negligible. It follows from Newton's second law of motion that $u=u(x, y, t)$ satisfies the wave
equation:
\begin{equation}\label{eq:wave}
u_{tt}=c^2 \Delta u  \  \text{on}\ \Omega,  \quad \ u=0 \ \text{on}\ \partial\Omega, 
\end{equation}
where $c=\sqrt{T/\rho}$. Solving this equation by separating the spacial and time  variables, we write $u(x, y,
t)=T(t)V(x, y)$. Then   
\[
\frac{T''}{c^2 T}=\frac{\Delta V}{V}=-\lambda,
\]
where $\lambda$ is a constant. The boundary value problem \eqref{eq:wave} is then reduced to the following Helmholtz equation with the Dirichlet boundary condition:
\begin{equation}\label{eq:lap}
\Delta V=-\lambda V \ \text{on}\ \Omega,  \quad V=0 \ \text{on}\ \partial\Omega.
\end{equation}
By Rellich's compactness lemma, the spectrum of the Dirichlet Laplacian consists of isolated eigenvalues $\lambda$ of finite multiplicity (see \cite[Theorem 6.3.1]{Davies95}). Let
\[
0<\lambda_1 \le \lambda_2\le \ldots \le
\lambda_k\le \ldots
\]
be the eigenvalues arranged in increasing order and repeated
according to their multiplicities. Let
$\varphi_k(x, y)$ be the eigenfunction associated with the eigenvalue $\lambda_k$.
Then the solution to the wave equation \eqref{eq:wave} has the form of
$$
u(x, y, t)=\sum_{k=1}^\infty \left(A_k\cos(c\sqrt{\lambda_k}t)+B_k\sin(c\sqrt{\lambda_k}t)\right)\varphi_k(x, y).
$$
The frequencies of the vibrations
$$
F_k=\frac{c\sqrt{\lambda_k}}{2\pi}, \quad k=1, 2, \ldots, 
$$
are constant multiples of the square root of the eigenvalues. 

The multiplicity of an eigenvalue $\lambda$ is the dimension of the associated eigenspace. Physically, the multiplicity represents the number of modes associated with the same frequency.  

In general, it is impossible to explicitly compute the eigenvalues for the Laplace operator, except for some special cases,
such as rectangles and tori.  The eigenvalues and associated eigenfunctions for the Dirichlet Laplacian on
a rectangle	$\Omega_{a, b} = \{(x,y) \ \mid \ \text{ 0}\leq x \leq a\text{, 0}\leq y \leq b\}$ can be easily computed by separation of variables, and they are given respectively by		
\[
		\begin{aligned}	
			\lambda_{m,n} 
			&=\pi ^2\left(\frac{m^2}{a^2}+
			\frac{n^2}{b^2}\right), \\
			V_{m,n}\text(x,y) &= \sin \frac{m\pi x}{a}\sin \frac{n\pi y}{b}, 
		\end{aligned} 
		\]
for $m,n \in \N$, where $\N$ is the set of all positive integers (see, for example, \cite[Lemma~6.2.1]{Davies95}). 
For any $\lambda >0$, the multiplicity $\mult(\lambda)$ of  $\lambda$ is the number of integer lattice points in the first quadrant on the following ellipse: 
		\begin{equation}\label{eq:mult-r}
		\frac{x^2}{(a\sqrt{\lambda}/\pi)^2} +\frac{y^2}{(b\sqrt{\lambda}/\pi)^2}=1.
		\end{equation}
We have the following simple observation (see~\cite{Fu2025-ph}):

\begin{lemma}\label{lm:1} 
If $(a/b)^2$ is irrational, then $\mult(\lambda_{m, n})=1$ for any  $m, n\in\N$. 
\end{lemma}

We now recall relevant well-known facts about eigenvalues and eigenfunctions for the Laplace operator on two-dimensional tori (see, for example, \cite{Milnor64}).  Let $L$ be the two dimensional lattice generated by  $\vec{a}=(a_1, \ a_2)$ and $\vec{b}=(b_1, b_2)$:
\begin{equation}\label{eq:L}
L = \{ m \vec{a} + n \vec{b} \mid m, n \in \mathbb{Z} \}
\end{equation}
Let
\[
M = 
\begin{bmatrix}
	a_1 & b_1 \\
	a_2 & b_2
\end{bmatrix}
\]
be the associated matrix with $D = \det(M) = a_1 b_2 - a_2 b_1 \neq 0$.  Let $\theta$ be the angle between $\vec{a}$ and $\vec{b}$. Replacing $\vec{b}$ by $-\vec{b}$ if necessary, we assume $\theta\in (0, \ \pi/2]$. The dual of \( L \) is then given by
\[
L^* = \{ w\in \mathbb{R}^2 \mid w \cdot z\in \mathbb{Z} ,  \quad \forall z \in L \}.
\]
The associated matrix for \( L^* \) is given by
\[
M_{\psi} = (M^{-1})^T = \frac{1}{D} \begin{bmatrix}
	b_2 & -a_2 \\
	-b_1 & a_1
\end{bmatrix} = [\vec{u}, \vec{v}].
\]

The eigenvalues for the Laplace operator on the torus $T=\R^2/L$ are 
\[
(2\pi)^2 \cdot |w|^2
\]
with associated eigenfunctions
\[
e(z) = \exp(2\pi i \, z \cdot w), \quad w \in L^*.
\]
Write  \( w = x\vec{u} + y\vec{v}, \; x, y \in \mathbb{Z} \). Then
\begin{align*}
	|w|^2  &= x^2|\vec{u}|^2 + 2xy \vec{u}\cdot\vec{v}+ y^2|\vec{v}|^2 \\
	&= x^2|\vec{b}|^2 - 2xy\vec{a}\cdot\vec{b}+ y^2|\vec{a}|^2 \\
	&= |\vec{b}|^2 \bigl( x^2 - 2r \cos \theta xy+ r^2 y^2\bigr),
\end{align*}
where \( r = |\vec{a}|/|\vec{b}|\).

Thus, the multiplicity of the eigenvalue \( (2\pi)^2 |w|^2 \) is the number of integral lattice points on the ellipse defined by
\begin{equation}\label{eq:quadratic}
\tilde f(x, y):=x^2 - 2 r \cos \theta xy + r^2 y^2= \frac{|w|^2}{|\vec{b}|^2}.
\end{equation}
Note that each non‑zero eigenvalue has multiplicity at least \( \geq 2 \). Analogous to Lemma~\ref{lm:1}, we have the following result: 

\begin{lemma}\label{lm:2}
	If the Laplacian on the torus  \( T=\R^2/L \) has an eigenvalue of multiplicity \( \geq 6 \), then both \( r^2 \) and \( r\cos(\theta) \) are rational.
\end{lemma}

\begin{proof}  Under the assumption, there exist at least three pairs of integral lattice points
\[
(x_1, y_1), \; (x_2, y_2),\; \text{and}\; (x_3, y_3) \; \text{with } 0 \leq y_1 < y_2 < y_3
\]
such that
\begin{equation}\label{eq:m}
\tilde f(x_1, y_1) = \tilde f(x_2, y_2) = \tilde f(x_3, y_3).
\end{equation}
Observe that
\[
\tilde f(x,y) = \bigl( x - r\cos\theta \, y \bigr)^2 + (r\sin\theta)^2 y^2.
\]
From \eqref{eq:m} we obtain
\begin{equation}\label{eq:sin}
	\begin{aligned}
(r\sin\theta)^2 &= \frac{(x_1 - r\cos\theta \, y_1)^2 - (x_2 - r\cos\theta \, y_2)^2}{y_2^2 - y_1^2}\\
&= \frac{(x_1 - r\cos\theta \, y_1)^2 - (x_3 - r\cos\theta \, y_3)^2}{y_3^2 - y_1^2}.
\end{aligned}
\end{equation}
This leads to
\begin{align*}
	&\left[ \frac{x_1 - x_2}{y_2 - y_1} + r \cos \theta \right] \left[ \frac{x_1 + x_2}{y_2 + y_1} - r \cos \theta \right] \\
	&\quad = \left[ \frac{x_1 - x_3}{y_3 - y_1} + r \cos \theta \right] \cdot \left[ \frac{x_1 + x_3}{y_3 + y_1} - r \cos \theta \right].
\end{align*}

Set
\[
s=\frac{x_1 - x_2}{y_2 - y_1} \cdot \frac{x_1 + x_2}{y_2 + y_1}
- \frac{x_1 - x_3}{y_3 - y_1} \cdot \frac{x_1 + x_3}{y_3 + y_1}
\]
and
\[
t=\frac{x_1 + x_3}{y_3 + y_1} - \frac{x_1 - x_3}{y_3 - y_1}
- \frac{x_1 + x_2}{y_2 + y_1} + \frac{x_1 - x_2}{y_2 - y_1}.
\]
Then
\begin{equation}\label{eq:cos}
(r\cos\theta) t=s.
\end{equation}
We claim that $t\not=0$.  We prove this by contradiction. Suppose $t=0$. Then
\begin{equation}\label{eq:colinear}
	\frac{x_3y_3 - x_1y_1}{y_3^2 - y_1^2} = \frac{x_2y_2 - x_1y_1}{y_2^2 - y_1^2}.
\end{equation}
Thus, the points $(x_iy_i,\ y_i^2)$, $i=1, 2, 3$ are collinear. Let $\lambda$ be the quotient given 
by \eqref{eq:colinear} and let $\beta=x_1y_1-\lambda y_1^2$.  Then 
\[
x_i=\lambda y_i+\beta/y_i, \quad i=1, 2, 3.
\]
Since all three points $(x_i,\ y_i)$, $i=1, 2, 3$ lie on the ellipse defined by \eqref{eq:quadratic}, $y_i$, $i=1, 2, 3$ are
the solutions to  
\[
A(\lambda y+\beta/y)^2+B(\lambda y+\beta/y)y+Cy^2=D
\]
where $A=1$, $B=-2r\cos\theta$, $C=r^2$, and $D=|w|^2/|\vec{b}|^2$.  The above equation reduces to
\begin{equation}\label{eq:n2}
(A\lambda^2+B\lambda+C)y^4+(2A\lambda\beta+B\beta-D)y^2+A\beta^2=0,
\end{equation}
which is a quadratic equation in $y^2$. (Note that the coefficient of $y^4$ in the above 
equation is positive because $B^2-4AC<0$.)  It is impossible for \eqref{eq:n2} to have 
three distinct positive solutions.  This concludes the proof of the claim.
  
From \eqref{eq:cos}, we then obtain that \( r \cos \theta \in \mathbb{Q} \).  It then follows from \eqref{eq:sin} that \( (r\sin\theta)^2 \in \mathbb{Q} \).  
Therefore
\[
r^2 = (r\cos\theta)^2 + (r\sin\theta)^2 \in \mathbb{Q}.
\]
\end{proof}

As in the previous section, the set \( \mathcal{M} (L)\) of multiplicities of non-zero eigenvalues of the Laplacian on tori is related to quadratic representations of integers. Let
\[
2r\cos\theta = \frac{\alpha}{\beta}, \qquad r^2 = \frac{\gamma}{\delta},
\]
where $\alpha, \beta, \gamma$, and $\delta$ are positive integers such that $\gcd(\alpha, \beta)=1$ and $\gcd(\gamma, \delta)=1$. Then
\[
\tilde{f}(x, y) = x^2 - \frac{\alpha}{\beta} xy + \frac{\gamma}{\delta} y^2
= \frac{1}{\beta\delta} \bigl( \beta\delta  x^2 - \alpha\delta xy + \gamma\beta y^2 \bigr).
\]
Note that
\[
\gcd(\beta\delta, \ \alpha\delta,\ \gamma\beta)=\gcd(\delta\cdot\gcd(\beta,\ \alpha),\ \gamma\beta)
=\gcd(\delta,\ \beta).
\]
Let $\tau=\gcd(\delta,\ \beta)$ and let
\[
f(x,y) := a x^2 + b xy + c y^2=\frac{1}{\tau}(\beta\delta x^2 - \alpha\delta xy + \gamma\beta y^2).
\]
Then $f$ is a primitive positive-definite binary quadratic form. In this case, the study of multiplicities of eigenvalues of the Laplace operator on the torus \( T=\R^2/L \) is reduced to studying 
the number of quadratic representations of positive integers by $f(x, y)$.

\subsection{Quadratic Representation of Integers}\label{s:prelim:q}

For completeness, we recall some basic definitions and facts regarding the quadratic representations of integers. We refer the reader to \cite[Chapter 3]{Niven-et-al} and \cite{Cox}
for background materials. 

A binary quadratic form 
\[
Q(x,y)=ax^2+bxy+cy^2
\]
is called \textit{primitive} if $\gcd(a,b,c)=1$ and \textit{positive-definite} if $Q(x,y)>0$ for $(x,y)\in\mathbb{Z}^2\setminus\{(0,0)\}$.

Much of our analysis revolves around the discriminant $\Delta:=b^2-4ac$ of $Q$. It is worth noting that if $Q$ is positive-definite, then $\Delta<0$. The discriminant $\Delta$ is called \textit{fundamental} if $\Delta\neq 1$, $p^2\nmid \Delta$ for any odd prime $p$, and $\Delta$ satisfies one of 
\[
 \begin{cases}
     \Delta\equiv 1\pmod{4}\\
     \Delta\equiv 8,12\pmod{16}
 \end{cases}.
\]
For any discriminant $\Delta$, there exists a unique pair of integers $\Delta_0$ and $f>0$ satisfying
\[
\Delta=\Delta_0f^2
\]
such that $\Delta_0$ is a fundamental discriminant. The integer $f$ is called the \textit{conductor} of $\Delta$. If the conductor $f\neq 1$, then $\Delta$ is called \textit{non-fundamental}.

The automorphism group $\text{Aut}(Q)$ associated to the quadratic form $Q$ is a subgroup of $\text{GL}_2(\mathbb{Z})$ consisting of 
\[
M=\begin{pmatrix}
    p&q\\
    r&s
\end{pmatrix}\in\text{GL}_2(\mathbb{Z})
\]
that satisfies
\[
Q(px+qy,rx+sy)=Q(x,y)
\]
for every $x,y\in\mathbb{Z}$. We call the elements in the subgroup $\text{Aut}^+(Q):=\text{SL}_2(\mathbb{Z})\cap \text{Aut}(Q)$ the \textit{proper automorphisms of $Q$}, and those in the subset $\text{Aut}^{-}(Q)=\text{Aut}(Q)\setminus \text{Aut}^+(Q)$ the \textit{improper automorphisms of $Q$}.

With such sets of automorphisms in mind, we define three counting functions $R_Q(n)$ which counts all representations of $n$ by $Q$, $r^+_Q(n)$ which counts all representations of $n$ by $Q$ \textit{up to proper automorphism}, and $r_Q(n)$ which counts all representations of $n$ by $Q$ \textit{up to any automorphism}.

Two forms $Q_1,Q_2$ of the same discriminant $\Delta$ belong to the same \textit{form class} if they are \textit{properly equivalent}, that is there exists some matrix 
\[
\begin{pmatrix}
    p&q\\
    r&s
\end{pmatrix}\in\text{SL}_2(\mathbb{Z})
\]
such that
\[
Q_1(px+qy,rx+sy)=Q_2(x,y)
\]
for all $x,y\in\mathbb{Z}$. The set of all form classes forms together with Gaussian composition form an abelian group $C(\Delta)$ called the \textit{form class group}. A form
\[
Q_1(x, y)=ax^2+bxy+cy^2
\]
is called \textit{ambiguous} if it is properly equivalent to 
\[
Q_2(x, y)=ax^2-bxy+cy^2,
\]
or equivalently, if its class in the form class group is of order $\leq 2$.

Let $K:=\mathbb{Q}(\sqrt{\Delta})$. We denote the ring of integers of $K$ as $\mathcal{O}_K$ and define the \textit{order of $\Delta$} to be
\[
\mathcal{O}:=\mathbb{Z}+f\mathcal{O}_K.
\]
The maximal order $\mathcal{O}_K$ is a Dedekind domain and thus possesses unique factorization of ideals into prime ideals. The order of $\Delta$, $\mathcal{O}$ may not be a Dedekind domain, but $\mathcal{O}$-ideals with absolute norm relatively prime to $f$ can still be factored uniquely into prime ideals.

We notate the set of proper fractional $\mathcal{O}$-ideals as $I(\mathcal{O})$ (which forms a group under ideal multiplication) and the set of principal $\mathcal{O}$-ideals as $P(\mathcal{O})$ (which forms a subgroup of the aforementioned group $I(\mathcal{O})$). The quotient group defined by 
\[C(\mathcal{O}):=I(\mathcal{O})/P(\mathcal{O})
\]
is called the \textit{ideal class group}. A class in the ideal class group is called \textit{ambiguous} if it is of order $\leq 2$.

The central result for our counting of quadratic representations is the isomorphism between the form class group $C(\Delta)$ and the ideal class group $C(\mathcal{O})$, proven as Theorem 7.7 in \cite{Cox}. The proof of Lemma \ref{pair ideal correspondence} should give the reader a sense of the shape of this more general correspondence. (Note that ambiguous forms are involutive with respect to Gaussian composition, and therefore classes of ambiguous forms correspond to ambiguous ideal classes.)

Specifically, we consider an immediate consequence of this isomorphism: there is a one-to-one correspondence between integral ideals in the ideal class associate with a form $Q$ with norm $n$ and solutions to $Q(x,y)=n$, up to proper automorphism (stated as Lemma 2.1 in \cite{BG06}).

We also assume throughout that every primitive, positive-definite form $Q$ represents infinitely many primes. This is a classical result of Heinrich Weber \cite{Weber1882}, accessible treatments of which can be found in \cite{Bri54,SW92}. This theorem possesses considerable significance in the history of the development of density theorems from Dirichlet's theorem to the Chebotarev density theorem.

\section{Multiplicities on Rectangles}\label{s:rectangles}

In this section, we prove Theorems~\ref{M} and~\ref{Q}. We begin with the following simple lemma.

\begin{lemma}\label{prime uniqueness}
For any $m,n\in\mathbb{N}$ coprime such that $mn>1$, any prime $p$ represented by the quadratic form
\[
Q(x,y):=mx^2+ny^2
\]
is represented uniquely, up to signs.
\end{lemma}
\begin{proof} Assume, for sake of contradiction, $p$ admits two representations 
\[
(x_1,y_1),(x_2,y_2) \in \mathbb{Z}^2
\]
such that $(x_1,y_1)\neq (\pm x_2,\pm y_2)$, that is 
\[
p=mx_1^2+ny_1^2=mx_2^2+ny_2^2.
\]

Taking the product of these two representations, we have
\[
p^2=(mx_1^2+ny_1^2)(mx_2^2+ny_2^2).
\]
By the generalized version of Brahmagupta's identity, we have a system of equalities
\begin{equation}\label{Brahmagupta1.eq}
    p^2=(mx_1x_2+ny_1y_2)^2+mn(x_1y_2-x_2y_1)^2
\end{equation}
\begin{equation}\label{Brahmagupta2.eq}
    p^2=(mx_1x_2-ny_1y_2)^2+mn(x_1y_2+x_2y_1)^2
\end{equation}
and a system of congruences 
\[\begin{cases}
    mx_1^2\equiv- ny_1^2 \pmod p\\
    mx_2^2\equiv- ny_2^2 \pmod p
\end{cases},\]
which when multiplied result in 
\[mx_1x_2\equiv\ \pm ny_1y_2 \pmod p\]

If $mx_1x_2\equiv\ - ny_1y_2 \pmod p$, then by equation \ref{Brahmagupta1.eq} we must have that 
\[
D:=x_1y_2-x_2y_1\equiv 0 \pmod p.
\]
Again, from equation \ref{Brahmagupta1.eq}, we have
\[mnD^2\leq p^2\]
and thus, by our assumption that $mn>1$
\[|D|\leq\dfrac{p}{\sqrt{mn}}<p.\]
The only integer multiple of $p$ satisfying this inequality is $0$, meaning $x_1y_2=x_2y_1$. By coprimality, this implies equality of the solutions up to signs, a contradiction.

For the case when $mx_1x_2\equiv ny_1y_2\pmod p$, a similar contradiction results from a near identical proof to the previous case, instead utilizing equation \ref{Brahmagupta2.eq}.
\end{proof}

\begin{lemma}\label{pair ideal correspondence}
Let $m,n \in \mathbb{N}$ be coprime such that $mn>3$. Consider the binary quadratic form
\[
Q(x,y) := mx^2 + ny^2.
\]

If $Q(x_1,y_1) = p^{2k-1}$ for some prime $p$, natural number $k$, and integers $(x_1,y_1)$, then 
\begin{enumerate}
    \item[i.] There exists an ideal 
    \[
    \mathfrak{p}_{1,k} = (p, mx_1 + y_1\sqrt{-mn})^{2k-1}
    \] in the ring of integers $\mathcal{O}_K$ of $K=\mathbb{Q}(\sqrt{-mn})$ such that $N(\mathfrak{p}_{1,k})=p^{2k-1}$ and $\mathfrak{p}_{1,k}$ belongs to the ideal class $C = [(m, \sqrt{-mn})]$.
    
    \item[ii.] If $(x_2,y_2) \in \mathbb{Z}^2$ is another solution such that $Q(x_2,y_2)=p^{2k-1}$ and $(x_2,y_2) \neq (\pm x_1, \pm y_1)$, then the ideal $\mathfrak{p}_{2,k}$ corresponding to $(x_2,y_2)$ is distinct from $\mathfrak{p}_{1,k}$.
\end{enumerate}
\end{lemma}

\begin{proof}
Let $K = \mathbb{Q}(\sqrt{-mn})$. We work with the element $\sigma = \sqrt{-mn}$, noting that $\sigma \in \mathcal{O}_K$ (even if it does not form an integral basis when $d \equiv 1 \pmod 4$).

To prove part (i), it suffices to prove the statement for $k=1$. Given a solution $Q(x_i,y_i) = p$, define the element $\beta_i \in \mathcal{O}_K$ by
\[
\beta_i = mx_i + y_i\sqrt{-mn}.
\]
Define the ideal $\mathfrak{p}_i = (p, \beta_i) \subseteq \mathcal{O}_K$.

We verify the norm by computing the ideal product $\mathfrak{p}_i \bar{\mathfrak{p}_i}$:
\[
\mathfrak{p}_i \bar{\mathfrak{p}_i} = (p, \beta_i)(p, \bar{\beta_i}) = (p^2, p\bar{\beta_i}, p\beta_i, \beta_i\bar{\beta_i}).
\]
Computing the norm of the generator $\beta_i$:
\[
N(\beta_i) = \beta_i\bar{\beta_i} = (mx_i)^2 - d y_i^2 = m^2x_i^2 + mn y_i^2 = m(mx_i^2 + ny_i^2).
\]
Since $mx_i^2 + ny_i^2 = p$, we have $N(\beta_i) = mp$. Substituting this back into the generators:
\[
\mathfrak{p}_i \overline{\mathfrak{p}_i} = (p^2, p\overline{\beta_i}, p\beta_i, mp) = p(p, \overline{\beta_i}, \beta_i, m).
\]
Since $\gcd(p,m)=1$, the ideal $(p, \dots, m)$ is the unit ideal $(1)$. Thus $\mathfrak{p}_i \overline{\mathfrak{p}_i} = (p)$, and taking norms gives $N(\mathfrak{p}_i) = p$.

We now establish $[\mathfrak{p}_i] = C$, the ideal class corresponding to the form which is represented by $\mathfrak{m} = (m, \sqrt{-mn})$.
The sum of ideals $\mathfrak{p}_i+\mathfrak{m}=(p,\beta_i,m,\sqrt{-mn})=(1)$ by the comprimality of $p$ and $m$. Therefore, since $\mathcal{O}_K$ is a Dedekind domain, we have
\[
\beta_i\in\mathfrak{p}_i\cap\mathfrak{m}=(\mathfrak{p}_i\mathfrak{m})(\mathfrak{p}_i+\mathfrak{m})=\mathfrak{p}_i\mathfrak{m}
\]
implying $(\beta_i)\subseteq \mathfrak{p}_i\mathfrak{m}$.

A check of norms confirms equality: $N((\beta_i)) = m p$ and $N(\mathfrak{p}_i \mathfrak{m}) = mp$. Since $(\beta_i) \subseteq \mathfrak{p}_i \mathfrak{m}$ and they have the same norm, $(\beta_i) = \mathfrak{p}_i \mathfrak{m}$.
In the class group, $[\mathfrak{p}_i][\mathfrak{m}] = [(\beta_i)] = 1$.
Since the form $Q$ is ambiguous, the class $C = [\mathfrak{m}]$ satisfies $C = C^{-1}$ via \cite{Gauss} \S 249. Thus $[\mathfrak{p}_i] = C$.

Suppose $(x_1,y_1)$ and $(x_2,y_2)$ are two solutions with corresponding ideals
\[
\mathfrak{p}_{1,k}=(p,\beta_1)^{2k-1}\]
\[\mathfrak{p}_{2,k}=(p,\beta_2)^{2k-1}
\]
and assume $\mathfrak{p}_{1,k}=\mathfrak{p}_{2,k}$. Then 
\[
\mathfrak{p}_{1,1}^{2k-1}=\mathfrak{p}_{2,1}^{2k-1}
\]
and 
\[
\mathfrak{p}_{1,1}^{2k-1}\mathfrak{p}_{2,1}^{-(2k-1)}=(1).
\]

Taking the class of the left hand side of this expression, we have 
\[
[\mathfrak{p}_{1,1}^{2k-1}\mathfrak{p}_{2,1}^{-(2k-1)}]=C^{2k}[\mathfrak{p}_{1,1}\mathfrak{p}_{2,1}^{-1}]=[\mathfrak{p}_{1,1}\mathfrak{p}_{2,1}^{-1}].
\]

Therefore, 
\[
[\mathfrak{p}_{1,1}\mathfrak{p}_{2,1}^{-1}]=1
\]
and, as we can assume $mn > 3$, 
\[
\mathfrak{p}_{1,1}=\mathfrak{p}_{2,1}.
\]
Naturally, define $\mathfrak{p}:=\mathfrak{p}_{1,1}=\mathfrak{p}_{2,1}$.

Recall $\beta_i = mx_i + y_i\sqrt{-mn}$. Since $\beta_i \in \mathfrak{p}_{i,1}$, we have:

\[
mx_i + y_i\sqrt{-mn} \equiv 0 \pmod{\mathfrak{p}}.
\]
This implies a congruence relation for $\sqrt{-mn}$ in the residue field $\mathcal{O}_K/\mathfrak{p} \cong \mathbb{F}_p$. Note that $y_i \not\equiv 0 \pmod p$. We can rearrange to solve for $\sqrt{-mn}$:
\[
\sqrt{-mn} \equiv -m x_i y_i^{-1} \pmod{\mathfrak{p}}.
\]
Since $\mathfrak{p}$ is fixed, the value of $\sqrt{-mn} \mod{\mathfrak{a}}$ is unique. Therefore:
\[
-m x_1 y_1^{-1} \equiv -m x_2 y_2^{-1} \pmod{p}.
\]
Since $\gcd(m,p)=1$, we can cancel $-m$ and cross-multiply:
\[
x_1 y_2 \equiv x_2 y_1 \pmod{p}.
\]
Thus, $p$ divides $D:=x_1 y_2 - x_2 y_1$.

The vectors $v_1,v_2$ defined by
\[
v_i:=\begin{bmatrix}
    x_i\sqrt{m}\\
    y_i\sqrt{n}
\end{bmatrix}
\]
lie on the circle of radius $\sqrt{p}$.

Notice that $D$ can be interpreted as the determinant of the matrix in $\frac{1}{\sqrt{mn}}[v_1,v_2]\in M_{2\times 2}(\mathbb{R})$, that is
\[
D=x_1y_2-x_2y_1=\dfrac{1}{\sqrt{mn}}\det\left(\begin{bmatrix}
    x_1\sqrt{m} & x_1\sqrt{m}\\
    y_1\sqrt{m} & y_2\sqrt{m}
\end{bmatrix}\right).
\] 

This allows us to bound $|D|$ as 
\[
|D|=\dfrac{p}{\sqrt{mn}}|\sin(\theta)|\leq \dfrac{p}{\sqrt{mn}},
\]
there $\theta$ is the angle between $v_1,v_2$. By assumption, $mn>1$ and $D \equiv 0 \pmod p$, we have $|D| < p$ meaning $D = 0$ so,
\[
x_1 y_2 = x_2 y_1.
\]
Since the pairs $(x_i, y_i)$ must be primitive, 
\[
(x_1, y_1) = \pm (x_2, y_2).
\]
\end{proof}

We can now prove Theorem~\ref{Q}.  Let $K:=\mathbb{Q}(\sqrt{-mn})$ with $\mathcal{O}_K$ its associated ring of integers and define the ambiguous binary quadratic form $Q$ as
\[
Q(x,y):=mx^2+ny^2.
\]

By Weber's Theorem and Lemma \ref{prime uniqueness} there exists some prime $p$ represented uniquely and primitively by $Q$. Therefore, Lemma \ref{pair ideal correspondence} there exists some ideal
\[
\mathfrak{p}\subseteq\mathcal{O}_K
\]
in the ideal class $C:=[(m,\sqrt{-mn})]\in\text{CL}(K)$. Further, this ideal $\mathfrak{p}$ has absolute norm
\[
N(\mathfrak{p})=p,
\]
implying $\mathfrak{p}$ is prime as an ideal in the Dedekind domain $\mathcal{O}_K$.

From \cite{Gauss} \S 249, as $Q$ is an ambiguous form, its corresponding ideal class $C:=[(m,\sqrt{-mn})]$ is ambiguous in the sense that 
\[
C^2=1,
\]
the primitive ideal class. Therefore the ideal class of the conjugate ideal $\bar{\mathfrak{p}}$ must also be $C$ as 
\[
[\mathfrak{p}][\bar{\mathfrak{p}}]=[\mathfrak{p}\bar{\mathfrak{p}}]=[(p)]=1.
\]

Let $\mathfrak{a}$ be an ideal in $C$ such that $N(\mathfrak{a})=p^{2k-1}$, where $k$ is as in the theorem. It follows that 
\[
\mathfrak{a}\bar{\mathfrak{a}}=(p^{2k-1})=(p)^{2k-1}=(\mathfrak{p}\bar{\mathfrak{p}})^{2k-1}=\mathfrak{p}^{2k-1}\bar{\mathfrak{p}}^{2k-1}.
\]
As $\mathcal{O}_K$ is a Dedekind domain, we therefore have 
\[
\mathfrak{a}=\mathfrak{p}^j\mathfrak{p}^{2k-1-j}
\]
for some $j\in\{0,\dots,2k-1\}$. Note that the norm of $\mathfrak{a}$ is invariant under this factorization.

Therefore, the ideal corresponding to any representation of $p^{2k-1}$ by $Q$ must be of the above form.

Further, note that conjugation of some $\mathfrak{a}_j$ gives, 
\[
\bar{\mathfrak{a}_j}=\overline{\mathfrak{p}^j\bar{\mathfrak{p}}^{2k-1-j}}=\mathfrak{p}^{2k-1-j}\bar{\mathfrak{p}}^{j}.
\]
Therefore, there are only $k$ choices for $j$ that lead to potentially distinct values of $\mathfrak{a}$. Indeed, distinct choices of $j$ correspond to distinct solutions to $p^{2k-1}=mx^2+ny^2$ by Lemma \ref{pair ideal correspondence}.  This concludes the proof of
Theorem~\ref{Q}.

Theorem~\ref{M} is then a direct consequence of Theorem~\ref{Q}. It follows from \eqref{eq:mult-r} that the multiplicity of the eigenvalue $\lambda$ is the number of lattice points in the first quadrant that are on the ellipse:
$$
\frac{\lambda}{\pi^2}=\frac{x^2}{a^2}+\frac{y^2}{b^2}.
$$
Applying our assumption that $(a/b)^2=n/m$ for some $m,n\in\mathbb{N}$ coprime, the ellipse equation becomes
\[\dfrac{\lambda}{\pi^2}=\dfrac{mx^2}{b^2n}+\dfrac{y^2}{b^2}.
\] 
Clearing the denominators, we have 
\[
\dfrac{b^2n\lambda}{\pi^2}=mx^2+ny^2.
\] 
As $\lambda$ is allowed to vary in $\mathbb{R}$, we may choose 
\[
\lambda=p^{2k-1}\cdot\dfrac{\pi^2}{b^2n},
\]
where $p$ is the prime guaranteed by Theorem \ref{Q} and $k$ is the desired multiplicity.

Note that the assumption that $mn>3$ is as required in Theorem \ref{Q} is safe, as the circular case received a thorough treatment in \cite{Fu2025-ph}.\qed

\section{Multiplicities on Tori}\label{s:tori}

We provide a proof of Theorems~\ref{Q2} and~\ref{N} in this section. 

\begin{lemma}\label{Ambiguous_Ideal}
    Every ambiguous ideal class in the ideal class group $C(\mathcal{O})$ contains an ideal invariant under conjugation.
\end{lemma}

\begin{proof}
    Let $C\in C(\mathcal{O})$ be an ambiguous ideal class. Let $\mathfrak{a}\in C$. By definition, there exists some $\alpha \in K^\times$ such that 
    \begin{equation}\label{eq:l1}
        \mathfrak{a}=\mathfrak{\bar{a}}\cdot \alpha\mathcal{O}.
    \end{equation}
    Conjugating equation \eqref{eq:l1}, substituting, and simplifying, we have 
    \[
        \mathcal{O}=\alpha\bar{\alpha}\mathcal{O},
    \]
    from which it follows that 
    \[
        \alpha\bar{\alpha}=u,
    \]
    for some unit $u$ of $\mathcal{O}$. Specifically, since $\alpha\bar{\alpha}=y_{K/\mathbb{Q}}(\alpha)\in\mathbb{Q}_{\geq 0}$ and $\Delta<0$, the only such choice is
    \begin{equation}\label{eq:l2}
        \alpha\bar{\alpha}=1.
    \end{equation}
    
    Applying Hilbert's Theorem 90 to equation \eqref{eq:l2}, we have \[\alpha=\beta/\bar{\beta}=b\beta/\overline{b\beta}\] for some $\beta\in K^\times$ and $b\in\mathbb{Z}$ such that $b\beta\in\mathcal{O}_K$, $b\beta\mathcal{O}$ is a principal fractional $\mathcal{O}$-ideal. Substituting this result back into equation \eqref{eq:l1}, we can define the sought ideal, \[\mathfrak{b}:=\mathfrak{a}\cdot\overline{b\beta}\mathcal{O}=\mathfrak{\bar{a}}\cdot b\beta\mathcal{O}\in C,\] which is in the appropriate class and is fixed under conjugation.
\end{proof}

\begin{lemma}\label{Stabilizers}
    Consider the group $\textup{Aut}^+(Q)$. For any point $\mathbf{x}\in\mathbb{Z}^2\setminus(0,0)$, \[\textup{Stab}(\mathbf{x})=\{I\}.\]
\end{lemma}

\begin{proof}
    This is a straightforward computation, considering the following table:
    
    \begin{center}
        \begin{tabular}{c|c}
        $\Delta$ & Elements of $\text{Aut}^+(Q)$ \\\hline
        $<-4$ & $\pm I$\\
        $-4$ & $\pm I,\pm \big(\begin{smallmatrix} 0&-1\\1&0\end{smallmatrix}\big)$\\
        $-3$ & $\pm I,\pm \big(\begin{smallmatrix} 0&-1\\1&1\end{smallmatrix}\big),\pm \big(\begin{smallmatrix} 1&1\\-1&0\end{smallmatrix}\big)$
        \end{tabular}
    \end{center}
\end{proof}

In what follows, We will prove Theorem~\ref{Q2} and the following corollary. We will then 
deduce Theorem~\ref{N} as a consequence of the corollary.  

\begin{corollary}\label{Raw_Counts}
    For the form $Q$, the representation counting function is as follows:
    \[R_Q(\mathbb{Z}_{\geq 0})=\begin{cases}
        \{1\}\cup 2\mathbb{Z}_{\geq 0}&\text{if}\;\Delta<-4\\
        \{1\}\cup 4\mathbb{Z}_{\geq 0}&\text{if}\;\Delta=-4\\
        \{1\}\cup 6\mathbb{Z}_{\geq 0}&\text{if}\;\Delta=-3\\
    \end{cases}.\]
\end{corollary}

\subsection{Proof of Theorem \ref{Q2}, Non-Ambiguous Forms}\label{Q_Non_Ambiguous}

\begin{proof}
Choose $p$ and $q$ prime, such that $p,q$ are represented by $Q$ and the principal form of $\Delta$, respectively, where both are coprime to $f$. Such primes are guaranteed to exist as a consequence of Weber's theorem on the density of primes in the image of certain quadratic forms (accessible treatments of which can be found in \cite{Bri54,SW92}).

By Theorem 7.7 of \cite{Cox}, there exist $\mathcal{O}$-ideals $\mathfrak{p}\in C_Q$ and $\mathfrak{q}\in C_1$ with norms $|\mathcal{O}/\mathfrak{p}|=p$ and $|\mathcal{O}/\mathfrak{q}|=q$, respectively. In terms of the conjugate ideals, these norm relationships can be expressed \begin{equation}\label{eq:1}
    \begin{cases}
        \mathfrak{p\bar{p}}=p\mathcal{O}\\
        \mathfrak{q\bar{q}}=q\mathcal{O}
    \end{cases}.
\end{equation}
Furthermore, $\mathfrak{p},\mathfrak{\bar{p}},\mathfrak{q},\mathfrak{\bar{q}}$ are prime because $p,q\nmid f$.

Fix $\alpha\geq 0$. We now show that there exists some $\mathcal{O}$-ideal in $C_Q$ with norm $pq^{\alpha}$. Specifically for $\beta\in\{0,\dots,\alpha\}$, the ideal $\mathfrak{p}\mathfrak{q}^\beta\bar{\mathfrak{q}}^{\alpha-\beta}$ has norm $|\mathcal{O}/\mathfrak{p}\mathfrak{q}^\beta\bar{\mathfrak{q}}^{\alpha-\beta}|=pq^\alpha$ and $\mathfrak{p}\mathfrak{q}^\beta\bar{\mathfrak{q}}^{\alpha-\beta}\in C_Q$. For the norm, we check \[N(\mathfrak{p}\mathfrak{q}^\beta\bar{\mathfrak{q}}^{\alpha-\beta})=N(\mathfrak{p})N(\mathfrak{q})^\beta N(\bar{\mathfrak{q}})^{\alpha-\beta}=pq^{\beta}q^{\alpha-\beta}=pq^{\alpha},\] and for class membership, we check \[[\mathfrak{p}\mathfrak{q}^{\beta}\mathfrak{\bar{q}}^{\alpha-\beta}]=[\mathfrak{p}][\mathfrak{q}]^\beta[\mathfrak{\bar{q}}]^{\alpha-\beta}=[\mathfrak{p}]=C_Q.\]

Now let $\mathfrak{a}\in C_Q$ be an $\mathcal{O}$-ideal with norm $|\mathcal{O}/\mathfrak{a}|=pq^{\alpha}$. That is,
\begin{equation}\label{eq:2}
    \mathfrak{a\bar{a}}=pq^\alpha\mathcal{O}.
\end{equation}
From equations \ref{eq:1} and \ref{eq:2}, we have
\begin{equation}\label{eq:3}
    \mathfrak{a\bar{a}}=p\mathcal{O}(q\mathcal{O})^\alpha=\mathfrak{p\bar{p}}\mathfrak{q}^\alpha\mathfrak{\bar{q}}^\alpha.
\end{equation}
Thus, as $\mathfrak{a}$ has norm $pq^{\alpha}$ and $(pq^{\alpha}, f)=1$, $\mathfrak{a}$ factors uniquely into prime $\mathcal{O}$-ideals, up to order. Therefore, from equation \ref{eq:3}, we have 
\begin{equation}\label{eq:4}
    \mathfrak{a}=\mathfrak{p}^{\gamma}\mathfrak{\bar{p}}^{1-\gamma}\mathfrak{q}^{\beta}\mathfrak{\bar{q}}^{\alpha-\beta}\quad\text{for some $(\gamma,\beta)\in\{0,1\}\times\{0,\dots,\alpha\}$},
\end{equation}
noting that the norm of $\mathfrak{a}$ is invariant under different factorizations.

Since $Q$ is non-ambiguous, $[\mathfrak{p}]\neq [\mathfrak{\bar{p}}]$ and we must choose $\gamma=1$ to ensure $\mathfrak{a}\in C_Q$. Therefore, only $\alpha+1$ choices of $(\gamma,\beta)$ are valid and thus there are $\alpha+1$ $\mathcal{O}$-ideals in $C_Q$ with norm $pq^\alpha$.

By the ideal--representation correspondence (Lemma 2.1 of \cite{BG06}), there are $\alpha+1$ distinct representations of $pq^{\alpha}$, up to proper automorphism. Letting $\alpha$ vary, we see that $r_Q^+$ is surjective.
\end{proof}

\subsection{Proof of Theorem \ref{Q2}, Ambiguous Forms}\label{Q_Ambiguous}

\begin{proof}
    Choose $q$ as in section \ref{Q_Non_Ambiguous}. Let $\mathfrak{a}$ be an ideal in $C_Q$ fixed by conjugation, as guaranteed by Lemma \ref{Ambiguous_Ideal} and define $a=|\mathcal{O}/\mathfrak{a}|$.

    Fix $\alpha\geq 0$. We now show that there exists some $\mathcal{O}$-ideal in $C_Q$ with norm $aq^{\alpha}$. Specifically, for $\beta\in\{0,\dots,\alpha\}$, the ideal $\mathfrak{a}\mathfrak{q}^\beta\bar{\mathfrak{q}}^{\alpha-\beta}$ has norm $|\mathcal{O}/\mathfrak{a}\mathfrak{q}^\beta\bar{\mathfrak{q}}^{\alpha-\beta}|=aq^\alpha$ and $\mathfrak{a}\mathfrak{q}^\beta\bar{\mathfrak{q}}^{\alpha-\beta}\in C_Q$. Checking these properties follows a similar argument to that in section \ref{Q_Non_Ambiguous}.

    Now let $\mathfrak{b}\in C_Q$ be an $\mathcal{O}$-ideal with norm $|\mathcal{O}/\mathfrak{b}|=aq^{\alpha}$. As $\mathfrak{a}$ and $\mathfrak{b}$ are in the same ideal class, there exists some $\gamma\in K^\times$ such that
    \begin{equation}\label{eq:7}
        \mathfrak{b}\mathfrak{a}^{-1}=\gamma\mathcal{O}.
    \end{equation}
    Conjugating both sides of equation \ref{eq:7} and multiplying gives 
    \[
        \gamma\bar{\gamma}\mathcal{O}=\mathfrak{a}^{-1}\mathfrak{\bar{a}}^{-1}\mathfrak{b}\mathfrak{\bar{b}}=(a\mathcal{O})^{-1}(aq^\alpha\mathcal{O})=(q\mathcal{O})^\alpha.
    \]
    Recalling $q\nmid f$, we further have
    \begin{equation}\label{eq:8}
        \gamma\mathcal{O}=\mathfrak{q}^{\beta}\mathfrak{\bar{q}}^{\alpha-
        \beta},
    \end{equation}
    for $\beta\in\{0,\dots, \alpha\}$.

    Thus $\mathfrak{b}=\mathfrak{a}\mathfrak{q}^{\beta}\mathfrak{\bar{q}}^{\alpha-\beta}$, and each one of the $\alpha+1$ choices for $\beta$ yields a unique ideal $\mathfrak{b}$ via unique factorization (its worth noting that $\mathfrak{a}$ may not be prime, but this does not pose an issue).

    By the ideal--representation correspondence (Lemma 2.1 of \cite{BG06}), there are $\alpha+1$ distinct representations of $aq^{\alpha}$, up to proper automorphism. Letting $\alpha$ vary, we see that $r_Q^+$ is surjective.
\end{proof}

\begin{proof}[Proof of Corollary~\ref{Raw_Counts}]
    Let $n>0$ and define $X_Q(n)$ to be the set of representations of $n$ by $Q$. Further, consider the action of $\text{Aut}^+(Q)$ on $X_Q(n)$. By the Orbit-Stabilizer Theorem and Lemma \ref{Stabilizers}, for any $\mathbf{x}\in X_Q(n)$, 
    \[|\text{Orb}(\mathbf{x})|=\dfrac{|\text{Aut}^+(Q)|}{|\text{Stab}(\mathbf{x})|}=\begin{cases}
        2&\text{if}\;\Delta<-4\\
        4&\text{if}\;\Delta=-4\\
        6&\text{if}\;\Delta=-3
    \end{cases}.\]
    Considering that $|\text{Stab}((0,0))|=\text{Aut}^+(Q)$, the claimed values of $R_Q(\mathbb{Z}_{\geq 0})$ follow.
\end{proof}

\subsection{Proof of Theorem~\ref{N}}

Proving the rational cases of Theorem~\ref{N} now amounts to applying the above Corollary~\ref{Raw_Counts} to the forms determined by the listed values of $r,\theta$.    Recall that $f$ is a primitive positive-definite binary quadratic form
    \[
    f(x,y)=\frac{1}{\tau}(\beta\delta x^2 - \alpha\delta xy + \gamma\beta y^2)
    \]
    with discriminant \[\Delta = \frac{1}{\tau^2}\left(\alpha^2\delta^2-4\beta^2\gamma\delta\right),\] where $\alpha/\beta=2r\cos\theta$, $\gamma/\delta=r^2$ with $(\alpha,\beta)$ and $(\gamma,\delta)$  pairwise coprime, and $\tau=\gcd(\beta, \delta)$.

    In terms of $r,\theta$, the discriminant can be expressed 
    \[
    \Delta = -4\left(\dfrac{r\cdot\beta\delta}{\gcd(\beta,\delta)}\right)^2 \sin^2\theta.
    \]

    It follows from Lemma~\ref{lm:2} that if $\mathcal{M}(T)$ is infinite, then both $r^2$ and $r\cos\theta$ are rational. 
    Suppose now that $r^2$ and $r\cos\theta$ are rational. Further suppose that $\Delta = -3$. By Corollary \ref{Raw_Counts}, proved above, it follows that 
    \[
    R_f(\mathbb{Z})=\{1\}\cup 6\mathbb{Z}_{\geq 0}
    \] and therefore $\mathcal{M}(T)=6\mathbb{N}.$

    Conversely, suppose that $\mathcal{M}(T)=6\mathbb{N}$. By Lemma \ref{lm:2}, it follows that $r^2$ and $r\cos\theta$ are both rational. Thus $f$ is well-defined. 
    \[
    R_f(\mathbb{Z})=\{1\}\cup 6\mathbb{Z}_{\geq 0}.
    \]
    Recalling that for any such form $f$, $\Delta\leq -3$, it follows that $ \Delta = -3.$

    Similar arguments suffice for the cases when $\Delta = -4$ and $\Delta < -4$. This concludes the
    proof of the first part of Theorem~\ref{N}.

    We now prove the irrational part of Theorem~\ref{N}. We will divide the proof into several lemmas. Write
    $$
    \tilde f(x, y)=x^2-2r\cos\theta xy +r^2y^2=x^2+bxy+cy^2.
    $$
    We then have $b^2-4c<0$, and at least one of $b$ or $c$ is irrational.  For any positive real number $z$, write  $R_{\tilde{f}}(z)$ the number of integer solution pairs $(x, y)$ for $z=\tilde f(x, y)$.  Note that because of the symmetry $\tilde f(x, y)=\tilde f(-x, -y)$, $R_{\tilde{f}}(z)$ is always an even number. Since $\{\tilde f(x, y) \mid x, y\in \mathbb{Z}\}$ is a countable set, there exists a positive real number $z$ such that $z=\tilde f(x, y)$ has no integer solutions. Moreover, by Lemma~\ref{lm:2}, $R_{\tilde{f}}(z)\le 4$. Let $R_{\tilde{f}}(\mathbb{R^{+}})=\{R_{\tilde{f}}(z) \mid z\in\mathbb{R}, z>0\}$.

    \begin{lemma}\label{lm:a} If $b\in\mathbb{Q}$ and $c\not\in\mathbb{Q}$, then $R_{\tilde{f}}(\mathbb{R^{+}})=\{0, 2, 4\}$. 
    \end{lemma}
    \begin{proof} Let $z$ be a positive real number. Suppose $(x, y)$ and $(x', y')$ are integer solutions to $z=\tilde f(x, y)$. Then
    $$
    x^2-x'^2+b(xy-x'y')=c(y'^2-y^2).
    $$
    It follows that 
	\[
    y'^2 = y^2 \quad\text{and}\quad 
	x'^2 + bx'y' = x^2 + bxy.
    \]
Therefore, if $(x, y)$ is an integer solution to $z=f(x, y)$, the complete set of solutions to $z=f(x, y)$ is:
	\[ S = \{ (x, y), \quad (-x-by, y), \quad (-x, -y), \quad (x+by, -y) \}. \]
These are integer solutions provided $by\in \mathbb{Z}$. Note that $|S|=2$ when either $y=0$ or $2x+by=0$. Furthermore, 
    if $z$ is in the form of $z=A+cB$ where $B$ is not a perfect square, then $z=f(x, y)$ has no integer solution. \end{proof}

    \begin{lemma}\label{lm:b}
Suppose $b\not\in\mathbb{Q}$ and $c\in\mathbb{Q}$.  Then
\[
R_{\tilde{f}}(\mathbb{R^{+}})=\begin{cases} \{0, 2\} & \text{$c$ is not a perfect square in $\mathbb{Q}$;}\\
\{0, 2, 4\} &\text{$c$ is a perfect square in $\mathbb{Q}$.}
\end{cases}
\]
\end{lemma}
\begin{proof} 	Let $z$ be a positive real number and let $(x, y)$ be an integer solution to $z=\tilde f(x, y)$. Then $z$ is uniquely decomposed into a rational part $M$ and an irrational part $L$:
	\[ z = (x^2 + cy^2) + b(xy)=M+bL \]
Therefore, any other integer solution $(x', y')$ to $z=f(x, y)$ must satisfy:
	\[x'y' = L \quad\text{and}\quad x'^2 + cy'^2 =M.\]
Substituting $y' = L/x'$ into the second equation above and multiplying both sides by $x'^2$, we have:
	\[ x'^4 - Mx'^2 + cL^2 = 0. \]
Applying the quadratic formula to solve for $x'^2$ and substituting $M = x^2 + cy^2$ and 
$L = xy$ into the discriminant, we obtain
\begin{align*}
 x'^2 &= \frac{M \pm \sqrt{M^2 - 4cL^2}}{2} \\
	&= \frac{(x^2 + cy^2) \pm (x^2 - cy^2)}{2}.
\end{align*}
Thus  $x'^2 = x^2$ or $x'^2 = cy^2$.

If $c$ is {\it not} the square of a rational number, then the only integer solutions are $(x, y)$ and $(-x, -y)$.  
Therefore, in this case, $R_{\tilde{f}}(\mathbb{R^{+}})=\{0, 2\}$.
	
If $c = r^2$ for some $r \in \mathbb{Q}$,  then the equation $z=f(x, y)$ has 4 distinct integer
solutions:
\[
\pm (x,\; y) \quad \text{and}\quad \pm (ry, \; x/r),
\]
provided both $ry$ and $x/r$ are integers, and $x\not=ry$.  Moreover,  the number of solutions is 2 if $x=ry$. 
Thus, in this case, $R_{\tilde{f}}(\mathbb{R}^+)=\{0, 2, 4\}$. \end{proof}

\begin{lemma}\label{lm:c}
Suppose $b, c\not\in\mathbb{Q}$ and $\{1, b, c\}$ is linearly independent over  $\mathbb{Q}$. Then $R_{\tilde{f}}(\mathbb{R^{+}})=\{0, 2\}$.
\end{lemma}

\begin{proof} Let $z$ be a positive real number, and let $(x, y)$ be an integer solution to $z=\tilde f(x, y)$.
Suppose $\{1, b, c\}$ is linearly independent over $\mathbb{Q}$. If $(x', y')$ is also an integer solution
to $z=\tilde f(x, y)$. Then $x'^2=x^2$, $x'y'=xy$, and $y'^2=y^2$. It follows that $(x', y')=\pm (x, y)$. Thus, in this case, $R_{\tilde{f}}(\mathbb{R^{+}})=\{0, 2\}$. 
\end{proof}

 We now deal with the case when $b, c\not\in\mathbb{Q}$, and $\{1, b, c\}$ is linearly dependent over $\mathbb{Q}$.
\begin{lemma}\label{lm:d}
Suppose $b, c\not\in\mathbb{Q}$ and $c=\alpha b+\beta$ for some $\alpha, \beta\in\mathbb{Q}$. Then 
\[
R_{\tilde{f}}(\mathbb{R^{+}})=\begin{cases} \{0, 2\},  & \text{if $\alpha^2+\beta$ is not a perfect square in $\mathbb{Q}$};\\
\{0, 2, 4\},  &\text{if $\alpha^2+\beta$ is a perfect square in $\mathbb{Q}$.}
\end{cases}
\]
\end{lemma}

\begin{proof} Note that $\alpha\not=0$. Since $b^2-4c<0$, we have 
\[
\alpha^2+\beta>\alpha^2-\alpha b+\frac{b^2}4=(\alpha-\frac{b}2)^2>0.
\]
\
Write 
\begin{equation}\label{m}
\tilde  f(x,y) = (x^2 + \beta y^2)+(xy + \alpha y^2)b.
\end{equation}
Suppose there are integer pairs $(x_1,y_1)$ and $(x_2,y_2)$ such that
\begin{equation}\label{mm}
\tilde  f(x_1,y_1) = \tilde f(x_2,y_2).
\end{equation}
Then 
\begin{align}
  x_1^2 + \beta y_1^2 &= x_2^2 + \beta y_2^2, \label{rp}\\
  x_1 y_1 + \alpha y_1^2 &= x_2 y_2 + \alpha y_2^2. \label{irp}
\end{align}
Let $d=y_2^2-y_1^2$. From \eqref{rp} and \eqref{irp}, we have
\begin{align}
  x_1^2 - x_2^2 &= \beta d \label{rp2}\\
  x_1 y_1 - x_2 y_2 &= \alpha d \label{rp3}
\end{align}
Multiplying both sides of \eqref{rp} by $y_1^2$ and substituting $x_1 y_1 = x_2 y_2 + \alpha d$, after simplification, we obtain
\begin{equation}\label{qq}
  d\,\big( x^2_2 + 2\alpha y_2x_2 + \alpha^2 d - \beta y^2_1 \big) = 0. 
\end{equation}
If $d=0$, then from \eqref{rp} and \eqref{irp}, we have $(x_1,y_1) = \pm (x_2, y_2)$.

We now assume that $d\not=0$. Then 
$$
x^2_2 + 2\alpha y_2x_2 + \alpha^2 d - \beta y^2_1 = 0. 
$$
Solving this for $x_2$ using the quadratic formula yields
\begin{equation}
  x_2 = -\alpha y_2 \pm|y_1|\sqrt{\alpha^2 + \beta}. \label{qf}
\end{equation}

Suppose $\alpha^2 + \beta$ is not the square of a rational number. Then $y_1\not=0$. Otherwise, if $y_1=0$, then $x_2 = -\alpha y_2$. Substituting this into \eqref{rp} gives
\begin{equation}
  x^2_1 = (\alpha^2+\beta) y^2_2.
\end{equation}
Because $x_1, y_2 \in \mathbb{Z}$, we get $x_1 = y_2 = 0$, which contradicts the assumption $d = y_2^2- y_1^2 \neq 0$.\par

Since $y_1 \neq 0$, we have $x_2\not\in\mathbb{Q}$. Therefore, in this case, integer solutions to equation \eqref{mm} must satisfy $(x_1, y_1)=\pm (x_2, y_2).$ 

Suppose $\alpha^2 + \beta$ is a rational square. If $y_1=0$, then $x_2=-\alpha y_2$. From~\eqref{rp}, 
we then have
\[
(x_2, \; y_2)=\pm(-\frac{\alpha x_1}{\sqrt{\alpha^2+\beta}}, \;  \frac{x_1}{\sqrt{\alpha^2+\beta}}).
\]
If we choose an integer $x_1$ so that both $x_2$ and $y_2$ are integers, then the equation $z=f(x_1, 0)$ has four integer solutions 
\[
\pm (x_1,\; 0) \quad \text{and} \quad \pm (-\frac{\alpha x_1}{\sqrt{\alpha^2+\beta}}, \; \frac{x_1}{\sqrt{\alpha^2+\beta}}).
\]

We now show that it is possible to choose $(x_1, y_1)\in\mathbb{Z}^2$ so that the only integer solutions to equation \eqref{mm} are $(x_2, y_2)=\pm (x_1, y_1)$. We choose $y_1>0$.  Using \eqref{qf} and \eqref{irp} to solve for $y_2$ and then $x_2$ in terms of $x_1$ and $y_1$, we have 
\begin{equation}\label{xy2}
(x_2, \; y_2)=\pm (-\frac{\alpha(x_1+\alpha y_1)}{\sqrt{\alpha^2+\beta}}+y_1, \, \frac{x_1+\alpha y_1}{\sqrt{\alpha^2+\beta}}).
\end{equation}
If we choose $(x_1, y_1)\in\mathbb{Z}^2$ such that 
\[
(\sqrt{\alpha^2+\beta}-\alpha) y_1=x_1\quad \text{or}\quad (\sqrt{\alpha^2+\beta}+\alpha)y_1=-x_1,
\]
then $y_2=y_1$ or $y_2=-y_1$. In this case, the only integer solutions to equation \eqref{mm} are $(x_2, y_2)=\pm (x_1, y_1)$.  We thus conclude the proof of the lemma.
\end{proof}

The second part of Theorem~\ref{N}, when either $r\cos\theta$ or $r^2$ is irrational, is then a direct consequence of Lemma~\ref{lm:2}, Lemma~\ref{lm:a},
Lemma~\ref{lm:b}, Lemma~\ref{lm:c}, and Lemma~\ref{lm:d}. This concludes the proof of Theorem~\ref{N}.

\begin{remark} It would be of interest to generalize Theorem~\ref{Q} and Theorem~\ref{N} to higher dimensions. Both theorems suggest that a domain whose set 
	of multiplicities is infinite is the exception rather than the norm. The collection of rectangles or tori with such a property has zero measure among their peers.  It would be interesting to quantify and understand this phenomenon in general. 
\end{remark}

%

\bigskip

\noindent{\bf Acknowledgments.}  Part of the work was done while the first author visited
the Institut des Hautes \'{E}tudes Scientifiques (IH\'{E}S) and Westlake University. He thanks these institutions for providing exceptional environments for research. The authors also thank Jack Heimrath for stimulating discussions on the subject.

\bibliographystyle{plain}
\bibliography{multiplicities.bib}

@unpublished{BerardHelffer,
	author = "B\'{e}rard, Pierre and Helffer, Bernard",
	title = "Upper bounds on eigenvalue multipicities for surfaces of genus $0$ revisited",
	
	note = "arXiv:2202.06587v4, December, 2024."
}

@article{Cheng,
	author ="Cheng, S.-Y." ,
	title = "Eigenfunctions and nodal sets",
	journal = "Math. Helv.",
	year = 1976,
	volume=51,
	pages="43--55",
}

@BOOK{Cox,
	TITLE = "Primes of the Form $x^2+ny^2$",
	SUBTITLE = {Fermat, Class Field Theory, and Complex Multiplication},
	AUTHOR = {David A. Cox},
	YEAR = {1989}, 
	PUBLISHER = {Wiley},
	ISBN = {978-1-4704-7028-9},
}

@Book{Gauss,
	AUTHOR      = {Gauss, Carl F.},
	TRANSLATOR  = {Clarke, Arthur A.},
	title       = {\it Disquisitiones Arithmetic\ae},
	language    = {langlatin},
	publisher   = {Springer-Verlag},
	location    = {New York},
	year        = {1986},
}

@ARTICLE{Fu2025-ph,
	title     = "Multiplicities of eigenvalues for the Laplace operator on a
	square",
	author    = "Fu, Siqi and Heimrath, Jack and Hsiao, Samuel",
	journal   = "Involve",
	publisher = "Mathematical Sciences Publishers",
	volume    =  18,
	number    =  3,
	pages     = "487--494",
	month     =  apr,
	year      =  2025,
	language  = "en"
}

@unpublished{Heimrath,
	author={Heimrath, Jack},
	title={Counting Multiplicities of Dirichlet Eigenvalues},
	school={Rutgers University--Camden},
	year={2023},
	note={Master's Thesis, Rutgers University--Camden}
}

@article{Kac1966,
	author = "Kac, Mark",
	title = "Can One Hear the Shape of a Drum?",
	journal = "The American Mathematical Monthly",
	year = 1966,
	pages="1--23",
	volume=73}

@book{McCartin,
	author = "McCartin, B.",
	title = "Laplacian Eigenstructure of the Equilateral Triangle",
	publisher = "Hikari Ltd",
	year = 2011
}

@article{Nad,
	author = "Nadirashvili, N.",
	title = "Multiple eigenvalues of the Laplace operator",
	journal = "Math USSR-Sb.",
	year = 1988,
	pages="225--238"
}

@article{Pinsky,
	author = "Pinsky, M.",
	title = "The eigenvalues of an equilateral triangle",
	journal = "SIAM J. Math. Anal",
	year = 1980,
	volume=11,
	nomber=5,
	pages="819--827"
}

@article{Weber1882,
	author = "Weber, H.",
	journal = {Mathematische Annalen},
	language = {ger},
	pages = {301-329},
	title = {Beweis des Satzes, dass jede eigentlich primitive quadratische Form unendlich viele Primzahlen darzustellen fähig ist},
	url = {http://eudml.org/doc/157037},
	volume = {20},
	year = {1882},
}

@BOOK{Niven-et-al,
	title     = "Introduction to the theory of numbers",
	author    = "Niven, Ivan and Zuckerman, Herbert and Montgomery, Hugh",
	publisher = "John Wiley \& Sons, Inc.",
	year      =  1991,
	language  = "en"
}

@BOOK{SW,
	title     = "Fourier Analysis on Euclidean Spaces",
	author    = "Stein, Elias M. and Weiss, Guido",
	publisher = "Princeton University Press",
	year      =  1971,
	language  = "en"
}

@article{BG06,
	author = {Valentin Blomer and Andrew Granville},
	title = {{Estimates for representation numbers of quadratic forms}},
	volume = {135},
	journal = {Duke Mathematical Journal},
	number = {2},
	publisher = {Duke University Press},
	pages = {261 -- 302},
	year = {2006},
	doi = {10.1215/S0012-7094-06-13522-6},
	URL = {https://doi.org/10.1215/S0012-7094-06-13522-6}
}

@article{Bri54, title={An Elementary Proof of a Theorem About the Representation of Primes by Quadratic Forms}, volume={6}, DOI={10.4153/CJM-1954-034-0}, journal={Canadian Journal of Mathematics}, author={Briggs, W. E.}, year={1954}, pages={353–363}}

@article{SW92,
	ISSN = {00029890, 19300972},
	URL = {http://www.jstor.org/stable/2325086},
	author = {Blair K. Spearman and Kenneth S. Williams},
	journal = {The American Mathematical Monthly},
	number = {5},
	pages = {423--426},
	publisher = {[Taylor & Francis, Ltd., Mathematical Association of America]},
	title = {Representing Primes by Binary Quadratic Forms},
	urldate = {2026-01-14},
	volume = {99},
	year = {1992}
}

@BOOK{Davies95,
	title     = "Spectral theory and differential operators",
	author    = "Davies, E.-B.",
	publisher = "Cambridge University Press",
	year      =  1995,
	language  = "en"
}

@article {Milnor64,
	AUTHOR = {Milnor, J.},
	TITLE = {Eigenvalues of the {L}aplace operator on certain manifolds},
	JOURNAL = {Proc. Nat. Acad. Sci. U.S.A.},
	FJOURNAL = {Proceedings of the National Academy of Sciences of the United
	States of America},
	VOLUME = {51},
	YEAR = {1964},
	PAGES = {542},      
}

\end{document}